\documentclass[a4paper]{amsart}
\usepackage[utf8]{inputenc}
\usepackage[T1]{fontenc}
\usepackage{lmodern}
\usepackage{color}
\usepackage{amsmath}

\usepackage{amssymb}
\usepackage[all]{xy}

\usepackage{microtype}
\usepackage[pdftitle={Braided quantum SU(2) groups},
  pdfauthor={Paweł Kasprzak, Ralf Meyer, Sutanu Roy and Stanisław Lech Woronowicz},
  pdfsubject={Mathematics}
]{hyperref}
\usepackage[lite]{amsrefs}
\newcommand*{\MRref}[2]{ \href{http://www.ams.org/mathscinet-getitem?mr=#1}{MR \textbf{#1}}}

\newcommand*{\arxiv}[1]{\href{http://www.arxiv.org/abs/#1}{arXiv: #1}}

\renewcommand{\PrintDOI}[1]{\href{http://dx.doi.org/\detokenize{#1}}{doi: \detokenize{#1}}}

\BibSpec{book}{%
  +{}  {\PrintPrimary}                {transition}
  +{,} { \textit}                     {title}
  +{.} { }                            {part}
  +{:} { \textit}                     {subtitle}
  +{,} { \PrintEdition}               {edition}
  +{}  { \PrintEditorsB}              {editor}
  +{,} { \PrintTranslatorsC}          {translator}
  +{,} { \PrintContributions}         {contribution}
  +{,} { }                            {series}
  +{,} { \voltext}                    {volume}
  +{,} { }                            {publisher}
  +{,} { }                            {organization}
  +{,} { }                            {address}
  +{,} { \PrintDateB}                 {date}
  +{,} { }                            {status}
  +{}  { \parenthesize}               {language}
  +{}  { \PrintTranslation}           {translation}
  +{;} { \PrintReprint}               {reprint}
  +{.} { }                            {note}
  +{.} {}                             {transition}
  +{,} { \PrintDOI}                   {doi}
  +{} { available at \url}            {eprint}
  +{}  {\SentenceSpace \PrintReviews} {review}
}

\DeclareMathOperator{\Obj}{\mathrm{Obj}}
\DeclareMathOperator{\C}{C}
\newcommand{\tp}{\mathbin{\xymatrix{*+<.7ex>[o][F-]{\scriptstyle\top}}}}
\usepackage{latexsym,amsthm}

\newcommand{\comments}[1]{}
\newcommand{\cst}{\textup{C}^*}
\newcommand{\id}{\textup{id}}
\newcommand{\I}{\textup{I}}
\newcommand{\Mor}{\operatorname{Mor}}

\newcommand{\M}{\operatorname{M}}

\newcommand{\tens}{\otimes}
\newcommand{\Tens}{\boxtimes}

\newcommand{\gbar}{\overline{g}}

\newcommand{\qbar}{\overline{q}}

\newcommand{\lambdabar}{\overline{\lambda}}

\newcommand{\gammatilde}{{\widetilde\gamma}}

\newcommand{\alphatilde}{{\widetilde \alpha}}

\newcommand{\zesp}{{\mathbb C}}

\newcommand{\torus}{{\mathbb T}}
\newcommand{\inte}{{\mathbb Z}}

\newcommand{\Nideal}{\mathfrak{N}}

\newcommand{\Vs}[1]{\rule{0mm}{#1mm}}
\newcommand{\vs}{\vspace{2mm}}

\newcommand{\compcent}[1]{\vcenter{\hbox{$#1\circ$}}}
\newcommand{\comp}{\mathbin{\mathchoice
{\compcent\scriptstyle}{\compcent\scriptstyle}
{\compcent\scriptscriptstyle}{\compcent\scriptscriptstyle}}}

\newcommand{\set}[2]{\left\{#1:#2\right\}}

\newcommand{\modul}[1]{\left|#1\right|}

\newcommand{\xxx}[1]{}

\numberwithin{equation}{section}
\newtheorem{Thm}{Theorem}[section]
\newtheorem{Lem}[Thm]{Lemma}
\newtheorem{Prop}[Thm]{Proposition}

\theoremstyle{definition}
\newtheorem{Def}[Thm]{Definition}
\theoremstyle{remark}

\newcommand{\etyk}[1]{\label{#1}\stepcounter{equation}\tag{\theequation}}

\newcommand*{\nb}{\nobreakdash}
\newcommand*{\SU}{\textup{SU}}
\newcommand*{\nSU}{\textup{U}}

\newcommand{\algebra}{{A}}

\newcommand{\GG}{\mathbb{G}}

\begin{document}
\title{Braided quantum SU(2) groups}

\author[Kasprzak]{Pawe\l{} Kasprzak}
\email{pawel.kasprzak@fuw.edu.pl}
\address{Katedra Metod Matematycznych Fizyki\\
  Wydział Fizyki, Uniwersytet Warszawski\\
  Pasteura 5\\02-093 Warszawa\\Poland}

\author[Meyer]{Ralf Meyer}
\email{rmeyer2@uni-goettingen.de}
\address{Mathematisches Institut\\
  Georg-August Universität Göttingen\\
  Bunsenstraße 3--5\\
  37073 Göttingen\\
  Germany}

\author[Roy]{Sutanu Roy}
\email{rmssutanu85@gmail.com}
\address{Department of Mathematics and Statistics\\
  University of Ottawa\\
  585 King Edward Avenue\\
  Ottawa, ON K1N 6N5\\
  Canada}

\author[Woronowicz]{Stanisław Lech Woronowicz}
\email{Stanislaw.Woronowicz@fuw.edu.pl}
\address{Instytut Matematyki\\Uniwersytet w Białymstoku, and\\
  Katedra Metod Matematycznych Fizyki\\
  Wydział Fizyki, Uniwersytet Warszawski\\
  Pasteura 5, 02-093 Warszawa, Poland}

\begin{abstract}
  We construct a family of $q$\nb-deformations of SU(2) for complex
  parameters $q\neq0$.  For real~$q$, the deformation coincides with
  Woronowicz' compact quantum $\SU_q(2)$ group.  For $q\notin
  \mathbb{R}$, $\SU_q(2)$ is only a braided compact quantum group
  with respect to a certain tensor product functor for
  \(\cst\)\nb-algebras with an action of the circle group.
\end{abstract}

\subjclass[2010]{81R50 (46L55, 46L06)}
\keywords{braided compact quantum group; $\SU_q(2)$; $\nSU_q(2)$.}
\thanks{S.~Roy was supported by a Fields--Ontario Postdoctoral
  fellowship.
  St.L.~Woronowicz was supported by the Alexander von Humboldt-Stiftung.}

\maketitle

\section{Introduction}
The $q$\nb-deformations of $\SU(2)$ for real deformation parameters
$0<q<1$ discovered in~\cite{Woronowicz:Twisted_SU2} are among the
first and most important examples of compact quantum groups.  Here we
construct a family of $q$\nb-deformations of~$\SU(2)$ for
\emph{complex} parameters $q\in\zesp^*=\zesp\setminus\{0\}$.
For $q\notin \mathbb{R}$, $\SU_q(2)$ is not a compact quantum group,
but a braided compact quantum group in a suitable tensor category.

A compact quantum group~$\mathbb{G}$ as defined
in~\cite{Woronowicz:CQG} is a pair $\mathbb{G} = (A,\Delta)$ where
$\Delta\colon A\rightarrow A\otimes A$ is a coassociative morphism
satisfying the cancellation law~\eqref{cancellation} below.  The
\(\cst\)\nb-algebra~$A$ is viewed as the algebra of continuous
functions on~$\GG$.

The theory of compact quantum groups is formulated within the
category~$\mathcal{C}^*$ of $\cst$-algebras.  This category with the
minimal tensor functor~$\otimes$ is a monoidal category
(see~\cite{MacLane:Categories}).  A more general theory may be
formulated within a monoidal category $(\mathcal{D}^*,\Tens)$,
where~$\mathcal{D}^*$ is a suitable category of $\cst$\nb-algebras
with additional structure and $\Tens\colon
\mathcal{D}^*\times\mathcal{D}^*\rightarrow\mathcal{D}^*$ is a
monoidal bifunctor on~$\mathcal{D}^*$.  Braided Hopf algebras may be
defined in braided monoidal categories (see
\cite{Majid:Quantum_grp}*{Definition 9.4.5}).  The braiding becomes
unnecessary when we work in categories of \(\cst\)\nb-algebras.

Let $A$ and~$B$ be $\cst$\nb-algebras.  The multiplier algebra
of~$B$ is denoted by~$\M(B)$.  A \emph{morphism} $\pi\in\Mor(A,B)$
is a $^*$\nb-homomorphism $\pi\colon A\rightarrow\M(B)$ with
$\pi(A)B = B$.  If \(A\) and~\(B\) are unital, a morphism is simply
a unital $^*$\nb-homomorphism.

Let~$\mathbb{T}$ be the group of complex numbers of modulus~$1$ and
let~$\mathcal{C}^*_{\mathbb{T}}$ be the category of
$\mathbb{T}$\nb-$\cst$-algebras; its objects are $\cst$\nb-algebras
with an action of~$\mathbb{T}$, arrows are
\(\mathbb{T}\)\nb-equivariant morphisms.  We shall use a family of
monoidal structures~$\Tens_{\zeta}$ on~$\mathcal{C}^*_{\mathbb{T}}$
parametrised by
$\zeta\in\mathbb{T}$, which is defined as
in~\cite{Meyer-Roy-Woronowicz:Twisted_tensor}.

The $\cst$\nb-algebra~$\algebra$ of~$\SU_q(2)$ is defined as the
universal unital $\cst$\nb-algebra generated by two elements
$\alpha,\gamma$ subject to the relations
\[
\etyk{SU2q}
\left\{
\begin{array}{r@{\;=\;}l}
\alpha^{*}\alpha+\gamma^{*}\gamma&\I,\\
\alpha\alpha^{*}+\modul{q}^{2}\gamma^{*}\gamma&\I,\\
\gamma\gamma^{*}&\gamma^{*}\gamma,\\
\alpha\gamma&\qbar\gamma\alpha,\\
\alpha\gamma^{*}&q\gamma^{*}\alpha.
\end{array}
\right.
\]

For real~$q$, the algebra~\(\algebra\) coincides with the algebra of
continuous functions on the quantum $\SU_q(2)$ group described
in~\cite{Woronowicz:Twisted_SU2}: $\algebra =\C(\SU_{q}(2))$.  Then
there is a unique morphism $\Delta\colon \algebra\to \algebra \tens
\algebra$ with
\[
\etyk{Delta}
\begin{array}{r@{\;=\;}l}
\Delta(\alpha)&\alpha\tens\alpha-q\gamma^{*}\tens\gamma,\\
\Delta(\gamma)&\gamma\tens\alpha+\alpha^{*}\tens\gamma.
\end{array}
\]
It is coassociative, that is,
\[
\etyk{coassociative}
(\Delta\tens\id_{\algebra})\comp\Delta
= (\id_{\algebra}\tens\Delta)\comp\Delta,
\]
and has the following cancellation property:
\[
\etyk{cancellation}
\begin{aligned}
  A\otimes A &= \Delta(A)(A\otimes\I),\\
  A\otimes A &= \Delta(A)(\I\otimes A);
\end{aligned}
\]
here and below, $EF$ for two closed subspaces $E$ and~$F$ of a
$\cst$\nb-algebra denotes the norm-closed linear span of the set of
products~$ef$ for $e\in E$, $f\in F$.

If~$q$ is not real, then the operators on the right hand sides
of~\eqref{Delta} do not satisfy the relations~\eqref{SU2q}, so there
is no morphism~$\Delta$ satisfying~\eqref{Delta}.  Instead,
\eqref{Delta} defines a \(\mathbb{T}\)\nb-equivariant morphism
\(\algebra\to \algebra\Tens_\zeta \algebra\) for the monoidal
functor~$\Tens_\zeta$ with $\zeta = q/\qbar$.  This morphism
in~$\mathcal{C}^*_{\mathbb{T}}$ satisfies appropriate analogues of
the coassociativity and cancellation laws \eqref{coassociative}
and~\eqref{cancellation}, so we get a braided compact quantum group.
Here the action of~\(\mathbb{T}\) on~\(A\) is defined by
\(\rho_z(\alpha)=\alpha\) and \(\rho_z(\gamma)=z\gamma\) for all
\(z\in\mathbb{T}\).

For $X,Y\in\Obj(\mathcal{C}^*)$, $X\tens Y$ contains commuting
copies $X\tens\I_{Y}$ of~$X$ and $\I_{X}\tens Y$ of~$Y$ with
\(X\tens Y=(X\tens\I_{Y})(\I_{X}\tens Y)\).  Similarly,
$X\Tens_\zeta Y$ for $X,Y\in\mathcal{C}^*_{\mathbb{T}}$ is a
$\cst$\nb-algebra with injective morphisms
$j_{1}\in\Mor(X,X\Tens_\zeta Y)$ and $j_{2}\in\Mor(Y,X\Tens_\zeta
Y)$ such that $X\Tens_\zeta Y = j_{1}(X)j_{2}(Y)$.  For
$\mathbb{T}$\nb-homogeneous elements $x\in X_k$ and $y\in Y_l$ (as
defined in~\eqref{homel}), we have the commutation relation
\[
\etyk{com}
j_{1}(x)j_{2}(y)=\zeta^{k l}j_{2}(y)j_{1}(x)
\]
The following theorem contains the main result of this paper:

\begin{Thm}
  \label{main}
  Let \(q\in\mathbb{C}\setminus\{0\}\) and \(\zeta=q/\qbar\).  Then
  \begin{enumerate}
  \item there is a unique \(\mathbb{T}\)\nb-equivariant morphism
    $\Delta\in\Mor(\algebra,\algebra\Tens_\zeta \algebra)$ with
    \[
    \etyk{Delta1}
    \begin{aligned}
      \Delta(\alpha)&=
      j_{1}(\alpha)j_{2}(\alpha)-qj_{1}(\gamma)^{*}j_{2}(\gamma),\\
      \Vs{5}\Delta(\gamma)&=
      j_{1}(\gamma)j_{2}(\alpha)+j_{1}(\alpha)^{*}j_{2}(\gamma);
    \end{aligned}
    \]
  \item $\Delta$ is coassociative, that is,
    \[
    (\Delta\Tens_\zeta\id_{\algebra})\comp\Delta
    = (\id_{\algebra}\Tens_\zeta\Delta)\comp\Delta;
    \]
  \item $\Delta$ obeys the cancellation law
    \[
    j_{1}(\algebra)\Delta(\algebra)
    = \Delta(\algebra)j_{2}(\algebra)
    = \algebra\Tens_\zeta \algebra.
    \]
  \end{enumerate}
\end{Thm}

We also describe some important features of the representation theory
of~$\SU_q(2)$ to explain the definition of~\(\SU_q(2)\), and we
relate~\(\SU_q(2)\) to the quantum \(\nSU(2)\) groups defined by Zhang
and Zhao in~\cite{Zhang-Zhao:Uq2}.

Braided Hopf algebras that deform \(\textup{SL}(2,\zesp)\) are
already described in~\cite{Majid:Examples_braided}.  We could,
however, find no precise relationship between Majid's braided Hopf
algebra \(\textup{BSL}(2)\) and our braided compact quantum
group~\(\SU_q(2)\).

\section{The algebra of \texorpdfstring{$\SU_{q}(2)$}{SUq(2)}}

The following elementary lemma explains what the defining
relations~\eqref{SU2q} mean:

\begin{Lem}
  \label{un}
  Two elements $\alpha$ and~$\gamma$ of a $\cst$\nb-algebra satisfy
  the relations~\eqref{SU2q} if and only if the following matrix is
  unitary:
  \[
  \begin{pmatrix}\alpha&-q\gamma^{*}\\\gamma&\alpha^{*}\end{pmatrix}
  \]
\end{Lem}

There are at least two ways to introduce a $\cst$\nb-algebra with
given generators and relations.  One may consider the
algebra~$\mathcal{\algebra}$ of all non-commutative polynomials in the
generators and their adjoints and take the largest $\cst$\nb-seminorm
on~$\mathcal{\algebra}$ vanishing on the given relations.  The
set~$\Nideal$ of elements with vanishing seminorm is an ideal
in~$\mathcal{\algebra}$.  The seminorm becomes a norm
on~$\mathcal{\algebra}/\Nideal$.  Completing
$\mathcal{\algebra}/\Nideal$ with respect to this norm gives the
desired $\cst$\nb-algebra~$A$.  Another way is to consider the
operator domain consisting of all families of operators satisfying the
relations.  Then~$\algebra$ is the algebra of all continuous operator
functions on that domain
(see~\cite{Kruszynski-Woronowicz:Gelfand-Naimark}).  Applying one of
these procedures to the relations~\eqref{SU2q} gives a
$\cst$\nb-algebra~$\algebra$ with two distinguished elements
$\alpha,\gamma\in \algebra$ that is universal in the following sense:

\begin{Thm}
  \label{universal}
  Let~$\widetilde{\algebra}$ be a $\cst$\nb-algebra with two elements
  $\alphatilde,\gammatilde\in\widetilde{\algebra}$ satisfying
  \[
  \etyk{SU2qtil}
  \left\{
    \begin{array}{r@{\;=\;}l}
      \tilde\alpha^{*}\tilde\alpha+\tilde\gamma^{*}\tilde\gamma&\I,\\
      \tilde\alpha\tilde\alpha^{*}+\modul{q}^{2}\tilde\gamma^{*}\tilde\gamma&\I,\\
      \tilde\gamma\tilde\gamma^{*}&\tilde\gamma^{*}\tilde\gamma,\\
      \tilde\alpha\tilde\gamma&\qbar\tilde\gamma\tilde\alpha,\\
      \tilde\alpha\tilde\gamma^{*}&q\tilde\gamma^{*}\tilde\alpha.
    \end{array}
  \right.
  \]
  Then there is a unique morphism
  $\rho\in\Mor(\algebra,\widetilde{\algebra})$ with
  \(\rho(\alpha)=\alphatilde\) and \(\rho(\gamma)=\gammatilde\).\qed
\end{Thm}

The elements $\alphatilde=\I_{\C(\mathbb{T})}\tens \alpha$ and
$\gammatilde=z\tens\gamma$ of $\C(\mathbb{T})\tens\algebra$
satisfy~\eqref{SU2qtil}.  Here $z\in\C(\mathbb{T})$ denotes the
coordinate function on~$\mathbb{T}$.  (Later, we also denote
elements of~\(\mathbb{T}\) by~\(z\).)  Theorem~\ref{universal} gives
a unique morphism
$\rho^A\in\Mor(\algebra,\C(\mathbb{T})\tens\algebra)$ with
\[
\etyk{rho}
\begin{array}{r@{\;=\;}r}
\rho(\alpha)&\I_{\C(\mathbb{T})}\tens\alpha,\\
\rho(\gamma)&z\tens\gamma.
\end{array}
\]
This is a continuous \(\mathbb{T}\)\nb-action, so we may view
$(A,\rho^A)$ as an object in the category~$\mathcal{C}^*_{\mathbb{T}}$
described in detail in the next section.

\begin{Thm}
  \label{the:compare_q}
  The $\cst$\nb-algebras~$\algebra$ for different~\(q\) with
  \(\modul{q}\neq0,1\) are isomorphic.
\end{Thm}

\begin{proof}
  During this proof, we write~\(\algebra_q\) for our
  \(\cst\)\nb-algebra with parameter~\(q\).

  First, \(\algebra_q \cong \algebra_{q'}\) for \(q'=
  q^{-1}\) by mapping \(\algebra_q\ni \alpha\mapsto \alpha'=
  \alpha^*\in \algebra_{q'}\) and \(\algebra_q\ni \gamma\mapsto
  \gamma' = q^{-1} \gamma \in \algebra_{q'}\).  Routine computations
  show that \(\alpha'\) and~\(\gamma'\) satisfy the
  relations~\eqref{SU2q}, so that Theorem~\ref{universal} gives a
  unique morphism \(\algebra_q\to\algebra_{q'}\) mapping
  \(\alpha\mapsto\alpha'\) and \(\gamma\mapsto\gamma'\).  Doing this
  twice gives \(q''=q\), \(\alpha''=\alpha\) and \(\gamma''=\gamma\),
  so we get an inverse for the morphism
  \(\algebra_q\to\algebra_{q'}\).  This completes the first step.  It
  reduces to the case \(0<\modul{q}<1\), which we assume from now on.

  Secondly, we claim that \(\algebra_q \cong \algebra_{\modul{q}}\) if
  \(0<\modul{q}<1\).  Equation~\eqref{SU2q} implies that~$\gamma$ is
  normal with \(\lVert\gamma\rVert \le 1\).  So we may use the
  functional calculus for continuous functions on the closed unit disc
  $\mathbb{D}^{1}=\set{\lambda\in\zesp}{\modul{\lambda}\leq 1}$.

  We claim that
  \[
  \etyk{Halmosh}
  \alpha f(\gamma)=f(\qbar\gamma)\alpha
  \]
  for all $f\in\C(\mathbb{D}^{1})$.  Indeed, the set $B\subseteq
  \C(\mathbb{D}^1)$ of functions satisfying~\eqref{Halmosh} is a
  norm-closed, unital subalgebra of~$\C(\mathbb{D}^{1})$.  The last
  two equations in~\eqref{SU2q} say that~$B$ contains the functions
  $f(\lambda)=\lambda$ and $f^*(\lambda)=\lambdabar$.  Since these
  separate the points of~$\mathbb{D}^{1}$, the Stone--Weierstrass
  Theorem gives $B=\C(\mathbb{D}^{1})$.

  Let $q=e^{i\theta}\modul{q}$ be the polar decomposition of~$q$.
  For $\lambda\in \mathbb{D}^{1}$, let
  \[
  g(\lambda)=
  \begin{cases}
    \lambda
    \textup{e}^{\textup{i}\theta\log_{\modul{q}}\modul{\lambda}}&
    \text{for }\lambda\neq 0,\\
    0&\text{for }\lambda= 0.
  \end{cases}
  \]
  This is a homeomorphism of~$\mathbb{D}^{1}$ because we get the map~\(g^{-1}\)
  if we replace~$\theta$ by~$-\theta$.  Thus $\gamma$ and
  $\gamma'=g(\gamma)$ generate the same $\cst$\nb-algebra.  We also
  get $g(\qbar\lambda)=\modul{q}g(\lambda)$, so inserting $f=g$ and
  $f=\gbar$ in~\eqref{Halmosh} gives
  \[
  a\gamma' =\modul{q}\gamma'\alpha,\qquad
  a(\gamma')^{*} = \modul{q} (\gamma')^{*}\alpha.
  \]
  Moreover, $\modul{g(\lambda)}=\modul{\lambda}$ and hence
  $\modul{\gamma'}=\modul{\gamma}$.  Thus we may replace~$\gamma$
  by~$\gamma'$ in the first three equations of~\eqref{SU2q}.  As a
  result, $\alpha$ and $\gamma'$ satisfy the relations~\eqref{SU2q}
  with~$\modul{q}$ instead of~$q$.  Since~\(g\) is a homeomorphism, an
  argument as in the first step now shows that $\algebra_q\cong
  \algebra_{\modul{q}}$.
  Finally, \cite{Woronowicz:Twisted_SU2}*{Theorem A2.2, page 180}
  shows that the \(\cst\)\nb-algebras~$\algebra_q$ for \(0<q<1\) are
  isomorphic.
\end{proof}

\section{Monoidal structure on
  \texorpdfstring{$\mathbb{T}$-$\cst$}{T-C*}-algebras}
\label{tcat}

We are going to describe the monoidal category
$(\mathcal{C}^*_{\mathbb{T}},\Tens_{\zeta})$ for
$\zeta\in\mathbb{T}$ that is the framework for our braided quantum
groups.  Monoidal categories are defined
in~\cite{MacLane:Categories}.

The \(\cst\)\nb-algebra \(\C(\mathbb{T})\) is a compact quantum
group with comultiplication
\[
\delta\colon \C(\mathbb{T}) \to \C(\mathbb{T}) \tens \C(\mathbb{T}),
\qquad
z\mapsto z\tens z.
\]

An object of $\mathcal{C}^*_{\mathbb{T}}$ is, by definition, a
pair~$(X,\rho^X)$ where~$X$ is a $\cst$\nb-algebra and
$\rho^{X}\in\Mor(X,\C(\mathbb{T})\tens X)$ makes the diagram
\[\etyk{dziacop}
\begin{gathered}
  \xymatrix{
    X\ar[rr]^{\rho^{X}}\ar[d]_{\rho^{X}}&&
    \C(\mathbb{T})\tens X\ar[d]^{\delta\tens\id}
    \\
    \C(\mathbb{T})\tens X
    \ar[rr]_-{\id_{\C(\mathbb{T})}\tens\rho^{X}}&&
    \C(\mathbb{T})\tens\C(\mathbb{T})\tens
    X }
\end{gathered}
\]
commute and satisfies the \emph{Podle\'s condition}
\[\etyk{Podl}
\rho^{X}(X)(\C(\mathbb{T})\tens\I_{X})=\C(\mathbb{T})\tens X.
\]
This is equivalent to a continuous $\mathbb{T}$\nb-action on~$X$ by
\cite{Soltan:Non_cpt_grp_act}*{Proposition 2.3}.

Let $X,Y$ be $\mathbb{T}$\nb-\(\cst\)-algebras.  The set of morphisms
from~$X$ to~$Y$ in~$\mathcal{C}^*_{\mathbb{T}}$ is the set
$\Mor_{\mathbb{T}}(X,Y)$ of \(\mathbb{T}\)\nb-equivariant morphisms
\(X\to Y\).  By definition, $\varphi\in\Mor(X,Y)$ is
\(\mathbb{T}\)\nb-equivariant if and only if the following diagram
commutes:
\[\etyk{tmor}
\begin{gathered}
  \xymatrix{
    X\ar[r]^-{\rho^{X}}\ar[d]_{\varphi}&
    \C(\mathbb{T})\tens
    X\ar[d]^{\id_{\C(\mathbb{T})}\tens\varphi}
    \\
    Y\ar[r]_-{\rho^{Y}}&
    \C(\mathbb{T})\tens Y }
\end{gathered}
\]

Let $X\in\mathcal{C}^*_{\mathbb{T}}$.  An element $x\in X$ is
\emph{homogeneous of degree $n\in\inte$} if
\begin{equation}
  \label{homel}
  \rho^{X}(x)=z^{n}\tens x.
\end{equation}
The degree of a homogeneous element~$x$ is denoted by
$\deg(x)$.  Let~$X_{n}$ be the set of homogeneous elements of~$X$ of
degree~$n$.  This is a closed linear subspace of~$X$, and
$X_{n}X_{m}\subseteq X_{n+m}$ and $X_{n}^{*}=X_{-n}$ for
\(n,m\in\inte\).  Moreover, finite sums of homogeneous elements are
dense in~$X$.

Let $\zeta\in\torus$.  The monoidal functor $\Tens_{\zeta}\colon
\mathcal{C}^*_{\mathbb{T}}\times
\mathcal{C}^*_{\mathbb{T}}\rightarrow \mathcal{C}^*_{\mathbb{T}}$ is
introduced as in~\cite{Meyer-Roy-Woronowicz:Twisted_tensor}.  We
describe $X\Tens_\zeta Y$ using quantum tori.  By definition, the
$\cst$\nb-algebra $\C(\torus^{2}_{\zeta})$ of the quantum torus is
the $\cst$\nb-algebra generated by two unitary elements~$U,V$
subject to the relation \(UV=\zeta\, VU\).

There are unique injective morphisms
$\iota_1,\iota_2\in\Mor(\C(\mathbb{T}),\C(\mathbb{T}^{2}_{\zeta}))$
with $\iota_1(z) = U$ and $\iota_2(z) = V$.  Define
$j_1\in\Mor(X,\C(\mathbb{T}^{2}_{\zeta})\otimes X\otimes Y)$ and
$j_2\in\Mor(Y,\C(\mathbb{T}^{2}_{\zeta})\otimes X\otimes Y)$ by
\begin{alignat*}{2}
  j_1(x) &= (\iota_1\otimes\id_X)\circ\rho^X(x)
  &\qquad &\text{for all }x\in X,\\
  j_2(y) &= (\iota_2\otimes\id_Y)\circ\rho^Y(y)
  &\qquad &\text{for all }y\in Y.
\end{alignat*}

Let $x\in X_k$ and $y\in Y_l$.  Then $j_1(x) = U^k\otimes x\otimes 1$ and
$j_2(y) = V^l\otimes 1\otimes y$, so that we get the commutation
relation~\eqref{com}.  This implies $j_1(X)j_2(Y) =j_2(Y) j_1(X)$, so
that $j_1(X)j_2(Y)$ is a $\cst$\nb-algebra.  We define
\[
X\Tens_{\zeta} Y = j_1(X)j_2(Y).
\]
This construction agrees with the one
in~\cite{Meyer-Roy-Woronowicz:Twisted_tensor} because
\(\C(\mathbb{T}^{2}_{\zeta}) \cong \C(\mathbb{T}) \Tens_{\zeta}
\C(\mathbb{T})\), see also the end of
\cite{Meyer-Roy-Woronowicz:Twisted_tensor}*{Section 5.2}.

There is a unique continuous $\mathbb{T}$\nb-action
$\rho^{X\Tens_\zeta Y}$ on~\(X\Tens_\zeta Y\) for which \(j_1\)
and~\(j_2\) are \(\mathbb{T}\)\nb-equivariant, that is,
$j_1\in\Mor_{\mathbb{T}}(X,X\Tens_{\zeta}Y)$ and
$j_2\in\Mor_{\mathbb{T}}(Y,X\Tens_{\zeta}Y)$.  This action is
constructed in a more general context
in~\cite{Meyer-Roy-Woronowicz:Twisted_tensor_2}.  We always
equip~$X\Tens_\zeta Y$ with this \(\mathbb{T}\)\nb-action and thus
view it as an object of~$\mathcal{C}^*_{\mathbb{T}}$.

The construction~\(\Tens_\zeta\) is a bifunctor; that is,
\(\mathbb{T}\)\nb-equivariant morphisms
$\pi_1\in\Mor_{\mathbb{T}}(X_1,Y_1)$ and
$\pi_2\in\Mor_{\mathbb{T}}(X_2,Y_2)$ induce a unique
\(\mathbb{T}\)\nb-equivariant morphism
$\pi_1\Tens_\zeta\pi_2\in\Mor_{\mathbb{T}}(X_1\Tens_\zeta
X_2,Y_1\Tens_\zeta Y_2)$ with
\begin{equation}
  \label{funct}
  (\pi_1\Tens_\zeta\pi_2)(j_{X_1}(x_1)j_{X_2}(x_2))
  = j_{Y_1}(\pi_1(x_1))j_{Y_2}(\pi_2(x_2))
\end{equation}
for all $x_1\in X_1$ and $x_2\in X_2$.

\begin{Prop}
  \label{8.44}
  Let $x\in X$ and $y\in Y$ be homogeneous elements.  Then
  \begin{align*}
    j_{1}(x)j_{2}(Y)&= j_{2}(Y)j_{1}(x),\\
    j_{1}(X)j_{2}(y)&= j_{2}(y)j_{1}(X).
  \end{align*}
\end{Prop}

\begin{proof}
  Equation~\eqref{com} shows that
  \[
  j_{1}(x)j_{2}(y)=j_{2}(y)j_{1}(\rho^X_{\zeta^{\deg(y)}}(x))
  \]
  for any $x\in X$ and any homogeneous $y\in Y$.
  Since~\(\rho^X_{\zeta^{\deg(y)}}\) is an automorphism of~\(X\),
  this implies \(j_1(X)j_2(y) = j_2(y) j_1(X)\).  Similarly,
  \(j_{1}(x)j_{2}(y) = j_{2}(\rho^Y_{\zeta^{\deg(x)}}(y))j_{1}(x)\)
  for homogeneous $x\in X$ and any $y\in Y$ implies
  \(j_{1}(x)j_{2}(Y) = j_{2}(Y)j_{1}(x)\).
\end{proof}

\section{Proof of the main theorem}

Let $\alpha$ and~$\gamma$ be the distinguished elements
of~$\algebra$.  Let $\alphatilde$ and~$\gammatilde$ be the elements
of $\algebra\Tens_\zeta \algebra$ appearing on the right hand side
of~\eqref{Delta1}:
\[
\etyk{Deltatilde}
\begin{array}{r@{\;=\;}l}
\alphatilde&j_{1}(\alpha)j_{2}(\alpha)-qj_{1}(\gamma)^{*}j_{2}(\gamma),\\
\Vs{5}\gammatilde&j_{1}(\gamma)j_{2}(\alpha)+j_{1}(\alpha)^{*}j_{2}(\gamma).
\end{array}
\]
We have $\deg(\alpha)=\deg(\alpha^{*})=0$, $\deg(\gamma)=1$ and
$\deg(\gamma^{*})=-1$ by~\eqref{rho}.  Assume \(\qbar\zeta=q\).
Using~\eqref{com} we may rewrite~\eqref{Deltatilde} in the following
form:
\[
\begin{array}{r@{\;=\;}l}
\alphatilde&j_{2}(\alpha)j_{1}(\alpha)-\qbar j_{2}(\gamma)j_{1}(\gamma)^{*},\\
\Vs{5}\gammatilde&j_{2}(\alpha)j_{1}(\gamma)+j_{2}(\gamma)j_{1}(\alpha)^{*}.
\end{array}
\]
Therefore,
\[
\etyk{Deltatilde1}
\begin{array}{r@{\;=\;}l}
\alphatilde^{*}&j_{1}(\alpha)^{*}j_{2}(\alpha)^{*}-q j_{1}(\gamma)j_{2}(\gamma)^{*},\\
\Vs{5}\gammatilde^{*}&j_{1}(\gamma)^{*}j_{2}(\alpha)^{*}+j_{1}(\alpha)j_{2}(\gamma)^{*}.
\end{array}
\]
The four equations \eqref{Deltatilde} and~\eqref{Deltatilde1}
together are equivalent to
\[
\etyk{mmm}
\begin{pmatrix}
  \alphatilde&-q\gammatilde^{*}\\\gammatilde&\alphatilde^{*}
\end{pmatrix}
=
\begin{pmatrix}
  j_{1}(\alpha)&-qj_{1}(\gamma)^{*}\\j_{1}(\gamma)&j_{1}(\alpha)^{*}
\end{pmatrix}
\begin{pmatrix}
  j_{2}(\alpha)&-qj_{2}(\gamma)^{*}\\j_{2}(\gamma)&j_{2}(\alpha)^{*}
\end{pmatrix}.
\]
Lemma~\ref{un} shows that the matrix
\[
\etyk{u}
u=
\begin{pmatrix}\alpha&-q\gamma^{*}\\\gamma&\alpha^{*}
\end{pmatrix}
\in\M_2(\algebra)
\]
is unitary.  Hence so is the matrix \(j_1(u)j_2(u)\) on the right
hand side of~\eqref{mmm}.  Now Lemma~\ref{un} shows that
$\alphatilde,\gammatilde\in \algebra\Tens_\zeta \algebra$
satisfy~\eqref{SU2qtil}.  So the universal property of~$\algebra$ in
Theorem~\ref{universal} gives a unique morphism~$\Delta$ with
$\Delta(\alpha)=\alphatilde$ and $\Delta(\gamma)=\gammatilde$.

The elements $\alpha$ and~$\gamma$ are homogeneous of degrees \(0\)
and~\(1\), respectively, by~\eqref{rho}.  Hence $\alphatilde$
and~$\gammatilde$ are homogeneous of degree \(0\) and~\(1\) as well.
Since \(\alpha\) and~\(\gamma\) generate~\(\algebra\), it follows
that~\(\Delta\) is \(\mathbb{T}\)\nb-equivariant.  This proves
statement~(1) in Theorem~\ref{main}.  Here we use the action
$\rho^{A\Tens_\zeta A}$ of~$\mathbb{T}$ with $\rho_z^{A\Tens_\zeta
  A} (j_1(a_1)j_2(a_2)) = j_1(\rho^A_z(a_1))j_2(\rho^A_z(a_2))$.  We
may rewrite~\eqref{mmm} as
\[
\begin{pmatrix}
  \Delta(\alpha)&-q\Delta(\gamma)^{*}\\\Delta(\gamma)&\Delta(\alpha)^{*}
\end{pmatrix}
= \begin{pmatrix}j_{1}(\alpha)&-qj_{1}(\gamma)^{*}\\j_{1}(\gamma)&j_{1}(\alpha)^{*}
\end{pmatrix}
\begin{pmatrix}j_{2}(\alpha)&-qj_{2}(\gamma)^{*}\\j_{2}(\gamma)&j_{2}(\alpha)^{*}
\end{pmatrix}.
\]
Identifying $\M_2(\algebra)$ with $\M_2(\zesp)\tens \algebra$, we may
further rewrite this as
\[
\etyk{Dmm3}
(\id\tens\Delta)(u) = (\id\tens j_{1})(u)\;(\id\tens j_{2})(u),
\]
where~$\id$ is the identity map on~$\M_2(\zesp)$.

Now we prove statement (2) in Theorem~\ref{main}.  Let
$j_{1},j_{2},j_{3}$ be the natural embeddings of~$\algebra$ into
$\algebra\Tens_\zeta \algebra\Tens_\zeta \algebra$.
Since~\(\Delta\) is \(\mathbb{T}\)\nb-equivariant, we may form
$\Delta\Tens_{\zeta}\id$ and $\id\Tens_{\zeta}\Delta$.  The values
of $\id\tens\left(\Delta\Tens_\zeta\id_{\algebra}\right)$ and
$\id\tens\left(\id_{\algebra}\Tens_\zeta\,\Delta\right)$ on the
right hand side of~\eqref{Dmm3} are equal:
\[
\begin{array}{r@{\;=\;(\id\tens j_{1})(u)\,(\id\tens j_{2})(u)\,(\id\tens j_{3})(u)}l}
\left(\id\tens(\Delta\Tens_\zeta\id_{\algebra})\circ\Delta\right)(u)&,\\
\left(\id\tens(\id_{\algebra}\Tens_\zeta\,\Delta)\circ\Delta\right)(u)&\Vs{5}.
\end{array}
\]
Thus $(\Delta\Tens_\zeta\id_{\algebra})\circ\Delta$ and
$(\id_{\algebra}\Tens_\zeta\,\Delta)\circ\Delta$ coincide on
$\alpha,\gamma,\alpha^{*},\gamma^{*}$.  Since the latter
generate~$\algebra$, this proves statement~(2) of
Theorem~\ref{main}.

Now we prove statement~(3).  Let
\[
S=\set{x\in \algebra}{j_{1}(x)\in\Delta(\algebra)j_{2}(\algebra)}.
\]
This is a closed subspace of~$\algebra$.  We may also
rewrite~\eqref{Dmm3} as
\[
\etyk{Dmm1}
\begin{pmatrix}
  j_{1}(\alpha)&-qj_{1}(\gamma)^{*}\\j_{1}(\gamma)&j_{1}(\alpha)^{*}
\end{pmatrix}
=
\begin{pmatrix}
  \Delta(\alpha)&-q\Delta(\gamma)^{*}\\\Delta(\gamma)&\Delta(\alpha)^{*}
\end{pmatrix}
\begin{pmatrix}
  j_{2}(\alpha)&-qj_{2}(\gamma)^{*}\\j_{2}(\gamma)&j_{2}(\alpha)^{*}
\end{pmatrix}^{*}.
\]
Thus $\alpha,\gamma,\alpha^{*},\gamma^{*}\in S$.  Let $x,y\in S$
with homogeneous~\(y\).  Proposition~\ref{8.44} gives
\begin{multline*}
  j_{1}(xy)
  =j_{1}(x)j_{1}(y)
  \in \Delta(\algebra)j_{2}(\algebra)j_{1}(y)
  = \Delta(\algebra)j_{1}(y)j_{2}(\algebra)
  \\\subseteq\Delta(\algebra)\Delta(\algebra)
  j_{2}(\algebra)j_{2}(\algebra)
  =\Delta(\algebra)j_{2}(\algebra).
\end{multline*}
That is, $xy\in S$.  Therefore, all monomials in
$\alpha,\gamma,\alpha^{*},\gamma^{*}$ belong to $S$, so that
$S=\algebra$.  Hence $j_{1}(\algebra)\subseteq
\Delta(\algebra)j_{2}(\algebra)$.  Now $\algebra\Tens_\zeta \algebra
= j_{1}(\algebra) j_{2}(\algebra) \subseteq \Delta(\algebra)
j_{2}(\algebra) j_{2}(\algebra) = \Delta(\algebra) j_{2}(\algebra)$,
which is one of the Podle\'s conditions.
Similarly, let
\[
R=\set{x\in \algebra}{j_{2}(x)\in j_{1}(\algebra)\Delta(\algebra)}.
\]
Then~$R$ is a closed subspace of~$\algebra$.  We may also
rewrite~\eqref{Dmm3} as
\[
\etyk{Dmm2}
\begin{pmatrix}
  j_{2}(\alpha)&-qj_{2}(\gamma)^{*}\\j_{2}(\gamma)&j_{2}(\alpha)^{*}\\
\end{pmatrix}
=
\begin{pmatrix}
  j_{1}(\alpha)&-qj_{1}(\gamma)^{*}\\j_{1}(\gamma)&j_{1}(\alpha)^{*}
\end{pmatrix}^{*}
\begin{pmatrix}
  \Delta(\alpha)&-q\Delta(\gamma)^{*}\\\Delta(\gamma)&\Delta(\alpha)^{*}
\end{pmatrix}.
\]
Thus $\alpha,\gamma,\alpha^{*},\gamma^{*}\in R$.  Let $x,y\in R$ with
homogeneous~\(x\).  Proposition~\ref{8.44} gives
\begin{multline*}
  j_{2}(xy)
  = j_{2}(x)j_{2}(y)\in j_{2}(x)j_{1}(\algebra)\Delta(\algebra)
  = j_{1}(\algebra) j_{2}(x) \Delta(\algebra)
  \\ \subseteq j_{1}(\algebra) j_{1}(\algebra) \Delta(\algebra)
  \Delta(\algebra)
  = j_{1}(\algebra)\Delta(\algebra).
\end{multline*}
Thus $xy\in R$.  Therefore, all monomials in
$\alpha,\gamma,\alpha^{*},\gamma^{*}$ belong to~$R$, so that
$R=\algebra$, that is, $j_{2}(\algebra)\subseteq j_{1}(\algebra)
\Delta(\algebra)$.  This implies $\algebra\Tens_\zeta \algebra =
j_{1}(\algebra) j_{2}(\algebra) \subseteq j_{1}(\algebra)
j_{1}(\algebra) \Delta(\algebra) = j_{1}(\algebra) \Delta(\algebra)$
and finishes the proof of Theorem~\ref{main}.

\section{The representation theory of
  \texorpdfstring{$\SU_{q}$}{SUq(2)}}

For real~\(q\), the relations defining the compact quantum
group~\(\SU_q(2)\) are dictated if we stipulate that the unitary
matrix in Lemma~\ref{un} is a representation and that a certain vector
in the tensor square of this representation is invariant.  Here we
generalise this to the complex case.  This is how we found $\SU_q(2)$.

Let~$\mathcal{H}$ be a $\mathbb{T}$\nb-Hilbert space, that is, a
Hilbert space with a unitary representation $U\colon
\mathbb{T}\rightarrow\mathcal{U}(\mathcal{H})$.  For $z\in\mathbb{T}$
and $x\in\mathcal{K}(\mathcal{H})$ we define
\[
\rho^{\mathcal{K}(\mathcal{H})}_z(x) = U_z xU_z^*.
\]
Thus $(\mathcal{K}(\mathcal{H}),\rho^{\mathcal{K}(\mathcal{H})})$ is a
\(\mathbb{T}\)\nb-\(\cst\)-algebra.  Let
$(X,\rho^X)\in\Obj(\mathcal{C}^*_{\mathbb{T}})$.  Since
$\rho^{\mathcal{K}(\mathcal{H})}$ is inner, the braided tensor product
$\mathcal{K}(\mathcal{H})\Tens_\zeta X$ may (and will) be identified
with $\mathcal{K}(\mathcal{H})\otimes X$ -- see
\cite{Meyer-Roy-Woronowicz:Twisted_tensor}*{Corollary 5.18} and
\cite{Meyer-Roy-Woronowicz:Twisted_tensor}*{Example 5.19}.

\begin{Def}
  \label{fd}
  Let $\mathcal{H}$ be a $\mathbb{T}$\nb-Hilbert space and let $v\in
  \M(\mathcal{K}(\mathcal{H})\tens \algebra)$ be a unitary element
  which is $\mathbb{T}$-invariant, that is,
  \((\rho^{\mathcal{K}(\mathcal{H})}_z\otimes\rho^X_z)(v) = v\).  We
  call~$v$ a \emph{representation} of~$\SU_{q}(2)$ on~$\mathcal{H}$ if
  \[
  (\id_{\mathcal{H}}\tens\Delta)(v)
  = (\id_{\mathcal{H}}\tens j_{1})(v)\;
  (\id_{\mathcal{H}}\tens j_{2})(v).
  \]
\end{Def}

Theorem~\ref{the:repr_SU_U} below will show that representations
of~\(\SU_q(2)\) are equivalent to representations of a certain
compact quantum group.  This allows us to carry over all the usual
structural results about representations of compact quantum groups
to~\(\SU_q(2)\).  In particular, we may tensor representations.  To
describe this directly, we need the following result:

\begin{Prop}
  \label{komutacja}
  Let $X,Y,U,T$ be $\mathbb{T}$\nb-$\cst$-algebras.  Let $v\in X\tens
  T$ and $w\in Y\tens U$ be homogeneous elements of degree~$0$.
  Denote the natural embeddings by
  \begin{alignat*}{2}
    i_{1}\colon X&\to X\Tens_\zeta Y,&\qquad
    i_{2}\colon Y&\to X\Tens_\zeta Y,\\
    j_{1}\colon U&\to U\Tens_\zeta T,&\qquad
    j_{2}\colon T&\to U\Tens_\zeta T.
  \end{alignat*}
  Then $(i_{1}\tens j_{2})(v)$ and $(i_{2}\tens j_{1})(w)$ commute in
  $(X\Tens_\zeta Y) \tens (U\Tens_\zeta T)$.
\end{Prop}

\begin{proof}
  We may assume that $v=x\tens t$ and $w=y\tens u$ for homogeneous
  elements $x\in X$, $t\in T$, $y\in Y$ and $u\in U$.  Since
  $\deg(v)=\deg(w)=0$, we get $\deg(x)=-\deg(t)$ and
  $\deg(y)=-\deg(u)$.  The following computation completes the proof:
  \begin{align*}
    & \phantom{{}={}}(i_{1}\tens j_{2})(v)\,(i_{2}\tens j_{1})(w)
    = \left(i_{1}(x)\tens j_{2}(t)\right)\left(i_{2}(y)\tens j_{1}(u)\right)
    \\&= i_{1}(x)i_{2}(y)\tens j_{2}(t)j_{1}(u)
    = \zeta^{\deg(x)\deg(y)-\deg(t)\deg(u)}i_{2}(y)i_{1}(x)\tens j_{1}(u)j_{2}(t)
    \\& = \left(i_{2}(y)\tens j_{1}(u)\right)\left(i_{1}(x)\tens j_{2}(t)\right)
    = (i_{2}\tens j_{1})(w)\;(i_{1}\tens j_{2})(v)
    \\&= (i_{2}\tens j_{1})(w)\;(i_{1}\tens j_{2})(v).\qedhere
  \end{align*}
\end{proof}

\begin{Prop}
  \label{tensprod}
  Let $\mathcal{H}_1$ and~$\mathcal{H}_2$ be $\mathbb{T}$\nb-Hilbert
  spaces and let $v_i\in \M(\mathcal{K}(\mathcal{H}_i)\otimes A)$
  for \(i=1,2\) be representations of~$\SU_q(2)$.  Define
  \[
  v = (\iota_1\otimes\id_\algebra)(v_1)
  (\iota_2\otimes\id_\algebra)(v_2)
  \in\M(\mathcal{K}(\mathcal{H}_1)\Tens_\zeta
  \mathcal{K}(\mathcal{H}_2)\otimes A)
  \]
  and identify $\mathcal{K}(\mathcal{H}_1)\Tens_\zeta
  \mathcal{K}(\mathcal{H}_2) \cong
  \mathcal{K}(\mathcal{H}_1\otimes\mathcal{H}_2)$.  Then $v \in
  \M(\mathcal{K}(\mathcal{H}_1\otimes\mathcal{H}_2)\otimes A)$ is a
  representation of~$\SU_q(2)$ on $\mathcal{H}_1\otimes
  \mathcal{H}_2$.  It is denoted $v_1\tp v_2$ and called the
  \emph{tensor product} of $v_1$ and~$v_2$.
\end{Prop}

\begin{proof}
  It is clear that~\(v\) is \(\mathbb{T}\)\nb-invariant.  We compute
  \begin{align*}
    (\id_{\mathcal{H}_1\otimes \mathcal{H}_2}\otimes\Delta)(v)
    &= (\id_{\mathcal{H}_1\otimes \mathcal{H}_2}\otimes\Delta)
    ((\iota_1\otimes\id_\algebra) (v_1)(\iota_2\otimes\id_\algebra)(v_2))\\
    &= (\iota_1\tens j_{1})(v_1)\;(\iota_1\tens j_{2})(v_1)\;
    (\iota_2\tens j_{1})(v_2)\;(\iota_2\tens j_{2})(v_2)
    \\&= (\iota_1\tens j_{1})(v_1)\;(\iota_2\tens j_{1})(v_2)\;
    (\iota_1\tens j_{2})(v_1)\; (\iota_2\tens j_{2})(v_2)
    \\&= (\id_{\mathcal{H}_1\otimes \mathcal{H}_2}\otimes j_1)(v)\;
    (\id_{\mathcal{H}_1\otimes \mathcal{H}_2}\otimes j_2)(v),
  \end{align*}
  where the third step uses Proposition~\ref{komutacja}.
\end{proof}

Now consider the Hilbert space~$\zesp^2$, let $\{e_0,e_1\}$ be its
canonical orthonormal basis.  We equip it with the representation
$U\colon \mathbb{T}\rightarrow\mathcal{U}(\zesp^2)$ defined by $U_z e_0 =
ze_0$ and $U_z e_1 = e_1$.  Let $\rho^{\M_2(\zesp)}$ be the action
implemented by~$U$:
\[
\rho^{\M_2(\zesp)}_z
\begin{pmatrix}
  a_{11}&a_{12}\\a_{21}&a_{22}
\end{pmatrix}
= \begin{pmatrix} a_{11}&za_{12} \\ \overline{z}a_{21} &a_{22}
\end{pmatrix},
\]
where $a_{ij}\in\zesp$.  We claim that
\[
u = \begin{pmatrix} \alpha&-q\gamma^{*}\\\gamma&\alpha^{*}
\end{pmatrix}
\in\M_2(\zesp)\otimes A
\]
is a representation of~$\SU_q(2)$ on~$\zesp^2$.  By Lemma~\ref{un},
the relations defining~$A$ are equivalent to~$u$ being unitary.  The
\(\mathbb{T}\)\nb-action on~\(A\) is defined so that~\(u\) is
\(\mathbb{T}\)\nb-invariant.  The comultiplication is defined exactly
so that~\(u\) is a corepresentation, see~\eqref{Dmm3}.

The particular shape of~$u$ contains further assumptions, however.  To
explain these, we consider an arbitrary compact quantum group
$\GG=(\C(\GG),\Delta_\GG)$ in~$\mathcal{C}^*_{\mathbb{T}}$ with a
unitary representation
\[
u = \begin{pmatrix} a&b\\c&d \end{pmatrix} \in \M_2(\C(\GG)),
\]
such that $a,b,c,d$ generate the $\cst$-algebra $\C(\GG)$.  We
assume that~$u$ is $\mathbb T$\nb-invariant for the above $\mathbb
T$\nb-action on~$\zesp^2$.  Thus $\deg(a) = \deg(d) =0$,
$\deg(b)=-1$, $\deg(c) =1$.

\begin{Thm}
  \label{the:SU_q_invariant_universal}
  Let~$\GG$ be a braided compact quantum group with a unitary
  representation~$u$ as above.  Assume $b\neq0$ and that the vector
  $e_0\otimes e_1-qe_1\otimes e_0 \in \zesp^2\otimes\zesp^2$ for
  \(q\in\zesp\) is invariant for the representation $u\tp u$.  Then
  \(q\neq0\), $\qbar\zeta=q$, $d = a^*$, $b = -qc^*$, and there is a
  unique morphism $\pi\colon \C(\SU_q(2)) \rightarrow\C(\GG)$ with
  $\pi(\alpha) =a$ and $\pi(\gamma) =c$.  This is
  $\mathbb{T}$\nb-equivariant and satisfies
  $(\pi\boxtimes_\zeta\pi)\circ\Delta_{\SU_q(2)} =
  \Delta_{\GG}\circ\pi$.
\end{Thm}

\begin{proof}
  The representation $u\tp u\in \M_4(\C(\GG))$ is given by
  Proposition~\ref{tensprod}, which uses a canonical isomorphism
  $\M_2(\zesp) \boxtimes_\zeta\M_2(\zesp) \cong \M_4(\zesp)$.  This
  comes from the following standard representation of
  $\M_2(\zesp)\boxtimes_\zeta \M_2(\zesp)$ on $\zesp^2\otimes
  \zesp^2$.  For $T,S\in \M_2(\zesp)$ of degree $k,l$ and
  $x,y\in\zesp^2$ of degree $m,n$, we let $\iota_1(T) \iota_2(S)
  (x\otimes y) = \overline\zeta^{l m}Tx\otimes Sy$.  By construction,
  $u\tp u$ is $(\iota_1\otimes\id_{\C(\GG)})(u)\cdot (\iota_2\otimes
  \id_{\C(\GG)})(u)$.  So we may rewrite the invariance of $e_0\otimes
  e_1-qe_1\otimes e_0$ as
  \begin{equation}
    \label{eq:invariant_vector}
    (\iota_1\otimes \id_{\C(\GG)})(u^*)(e_0\otimes e_1-qe_1\otimes e_0)
    = (\iota_2\otimes \id_{\C(\GG)})(u)(e_0\otimes e_1-qe_1\otimes e_0)
  \end{equation}
  in $\zesp^2\otimes\zesp^2\otimes \C(\GG)$.  The left and right hand sides
  of~\eqref{eq:invariant_vector} are
  \begin{gather*}
    e_0\otimes e_1 \otimes a^*+e_1\otimes e_1\otimes b^* - q e_0\otimes
    e_0\otimes c^* - qe_1\otimes e_0\otimes d^*,\\
    e_0\otimes e_0\otimes b + e_0\otimes e_1\otimes d
    - qe_1\otimes e_0\otimes a
    - q \overline{\zeta} e_1\otimes e_1\otimes c,
  \end{gather*}
  respectively.  These are equal if and only if $b = -q c^*$, $d =
  a^*$, and $b^* = -q\overline\zeta c$.  Since $b\neq0$, this implies
  \(q\neq0\) and $\qbar\zeta=q$, and~\(u\) has the form in
  Lemma~\ref{un}.  Since~$u$ is a representation, it is unitary.  So
  $a,c$ satisfy the relations defining $\SU_q(2)$ and
  Theorem~\ref{universal} gives the unique morphism~$\pi$.  The
  conditions on~$u$ in Definition~\ref{fd} imply that~$\pi$ is
  $\mathbb{T}$\nb-equivariant and compatible with comultiplications.
\end{proof}

The proof also shows that~\(q\) is uniquely determined by the
condition that \(e_0\otimes e_1 -qe_1\otimes e_0\) should be
\(\SU_q(2)\)\nb-invariant.  Up to scaling, the basis \(e_0,e_1\) is the
unique one
consisting of joint eigenvectors of the \(\mathbb{T}\)\nb-action
with degrees \(1\) and~\(0\).  Hence the braided quantum
group \((\C(\SU_q(2)),\Delta)\) determines~\(q\) uniquely.

An invariant vector for~\(\SU_q(2)\) should also be homogeneous for the
\(\mathbb{T}\)\nb-action.  There are three cases of homogeneous
vectors in \(\zesp^2\otimes\zesp^2\): multiplies of \(e_0\otimes
e_0\), multiples of \(e_1\otimes e_1\), and linear combinations of
\(e_0\otimes e_1\) and \(e_1\otimes e_0\).  If a non-zero multiple
of~\(e_i\otimes e_j\) for \(i,j\in\{0,1\}\) is invariant, then the
representation~$u$ is reducible.  Ruling out such degenerate cases, we
may normalise the invariant vector to have the form \(e_0\otimes e_1
-qe_1\otimes e_0\) assumed in
Theorem~\ref{the:SU_q_invariant_universal}.

Roughly speaking, $\SU_q(2)$ is the universal family of braided
quantum groups generated by a $2$\nb-dimensional representation with
an invariant vector in \(u\tp u\).

There is, however, one extra symmetry that changes the
\(\mathbb{T}\)\nb-action on~\(\C(\SU_q(2))\) and that corresponds to the
permutation of the basis \(e_0,e_1\).  Given a
\(\mathbb{T}\)\nb-algebra~\(A\), let \(S(A)\) be the same
\(\cst\)\nb-algebra with the \(\mathbb{T}\)\nb-action by
\(\rho^{S(A)}_z = (\rho^A_z)^{-1}\).  Since the commutation
relation~\eqref{com} is symmetric in \(k,l\), there is a unique
isomorphism
\[
S(A\boxtimes_\zeta B) \cong S(A) \boxtimes_\zeta S(B),\qquad
j_1(a)\mapsto j_1(a),\quad j_2(b)\mapsto j_2(b).
\]
Hence the comultiplication on~\(\C(\SU_q(2))\) is one
on~\(S(\C(\SU_q(2)))\) as well.

\begin{Prop}
  \label{pro:Aq_symmetry}
  The braided quantum groups \(S(\C(\SU_q(2)))\)
  and \(\C(\SU_{\tilde{q}}(2))\) for
  \(\tilde{q} = \qbar^{-1}\) are isomorphic as braided quantum
  groups.
\end{Prop}

\begin{proof}
  Let \(\alpha,\gamma\) be the standard generators of
  \(A_q = \C(\SU_q(2))\) and
  let \(\tilde{\alpha},\tilde{\gamma}\) be the standard generators
  of~\(A_{\tilde{q}}\).  We claim that there is an isomorphism
  \(\varphi\colon A_q\to A_{\tilde{q}}\) that maps \(\alpha\mapsto
  \tilde\alpha^*\) and \(\gamma\mapsto \tilde{q}\tilde\gamma^*\) and
  that is an isomorphism of braided quantum groups from \(S(A_q)\)
  to~\(A_{\tilde{q}}\).  Lemma~\ref{un} implies that the matrix
  \[
  \begin{pmatrix} 0&1\\-1&0 \end{pmatrix}
  \begin{pmatrix}
    \tilde\alpha&-\tilde{q}\tilde\gamma^*\\
    \tilde\gamma&\tilde\alpha^*
  \end{pmatrix}
  \begin{pmatrix} 0&-1\\1&0 \end{pmatrix}
  =
  \begin{pmatrix}
    \tilde\alpha^*&-\tilde\gamma\\
    \tilde{q}\tilde\gamma^*&\tilde\alpha
  \end{pmatrix}
  =
  \begin{pmatrix}
    \varphi(\alpha)&\varphi(-q \gamma^*)\\
    \varphi(\gamma)&\varphi(\alpha^*)
  \end{pmatrix}
  \]
  is unitary.  Now Lemma~\ref{un} and Theorem~\ref{universal} give
  the desired morphism~\(\varphi\).  Since the inverse
  of~\(\varphi\) may be constructed in the same way, \(\varphi\) is
  an isomorphism.  On generators, it reverses the grading, so it is
  \(\mathbb{T}\)\nb-equivariant as a map \(S(A_q)\to
  A_{\tilde{q}}\).

  Let \(\Delta\) and \(\tilde\Delta\) denote the comultiplications
  on \(S(A_q)\) and~\(A_{\tilde{q}}\).  We compute
  \begin{align*}
    (\varphi\boxtimes_\zeta\varphi)\Delta(\alpha)
    &= (\varphi\boxtimes_\zeta\varphi) (j_1(\alpha)j_2(\alpha) -
    qj_1(\gamma^*) j_2(\gamma))
    \\&= j_1(\varphi(\alpha)) j_2(\varphi(\alpha))
    - qj_1(\varphi(\gamma^*)) j_2(\varphi(\gamma))
    \\&= j_1(\tilde\alpha^*) j_2(\tilde\alpha^*)
    - \tilde{q} j_1(\tilde\gamma) j_2(\tilde\gamma^*),\\
    \tilde\Delta(\varphi(\alpha))
    &= \tilde\Delta(\tilde\alpha^*)
    = j_2(\tilde\alpha)^* j_1(\tilde\alpha)^* -
    q^{-1} j_2(\tilde\gamma)^* j_1(\tilde\gamma)
    \\&= j_1(\tilde\alpha)^* j_2(\tilde\alpha)^* -
    q^{-1} \zeta j_1(\tilde\gamma) j_2(\tilde\gamma)^*.
  \end{align*}
  These are equal because \(\tilde{q} = \qbar^{-1} =
  q^{-1}\zeta\).  Similarly,
  \((\varphi\boxtimes_\zeta\varphi)\Delta(\gamma)
  =\tilde\Delta(\varphi(\gamma))\).  Thus~\(\varphi\) is an
  isomorphism of braided quantum groups.
\end{proof}

\section{The semidirect product quantum group}
\label{sec:U2}

A quantum analogue of the semidirect product construction for groups
turns the braided quantum group~$\SU_q(2)$ into a genuine compact
quantum group $(B,\Delta_B)$; we will publish details of this
construction separately.  Here~\(B\) is the universal
\(\cst\)\nb-algebra with three generators $\alpha,\gamma,z$ with the
$\SU_q(2)$-relations for $\alpha$ and~$\gamma$ and
\begin{align*}
  z\alpha z^* &=\alpha,\\
  z\gamma z^* &= \zeta^{-1}\gamma,\\
  z z^* &=z^*z = \I;
\end{align*}
the comultiplication is defined by
\begin{align*}
  \Delta_B(z)&= z\otimes z,\\
  \Delta_B(\alpha) &= \alpha\otimes\alpha-q\gamma^*z\otimes\gamma,\\
  \Delta_B(\gamma) &=\gamma\otimes\alpha+\alpha^*z\otimes\gamma.
\end{align*}
There are two embeddings $\iota_1, \iota_2\colon A \rightrightarrows
B\otimes B$ defined by
\begin{alignat*}{2}
  \iota_1(\alpha) &= \alpha\otimes \I&\qquad
  \iota_2(\alpha) &= \I\otimes\alpha, \\
  \iota_1(\gamma) &= \gamma\otimes \I&\qquad
  \iota_2(\gamma) &= z\otimes\gamma.
\end{alignat*}
Homogeneous elements $x,y\in A$ satisfy
\[
\etyk{SU2_commutation}
\iota_1(x)\iota_2(y) = \zeta^{\deg(x)\deg(y)}\iota_2(y)\iota_2(x).
\]
Thus we may rewrite the comultiplication as
\begin{align*}
  \Delta_B(z)&= z\otimes z,\\
  \Delta_B(\alpha) &=
  \iota_1(\alpha)\iota_2(\alpha)-q\iota_1(\gamma)^*\iota_2(\gamma),\\
  \Delta_B(\gamma) &=
  \iota_1(\gamma)\iota_2(\alpha)+\iota_1(\alpha)^*\iota_2(\gamma).
\end{align*}
In particular, $\Delta_B$ respects the commutation relations for
$(\alpha,\gamma,z)$, so it is a well-defined
$^*$\nb-\hspace{0pt}homomorphism \(B\to B\otimes B\).  It is routine
to check the cancellation conditions~\eqref{cancellation} for \(B\),
so $(B,\Delta_B)$ is a compact quantum group.

This is a compact quantum group with a projection as
in~\cite{Roy:Qgrp_with_proj}.  Here the projection $\pi\colon
B\rightarrow B$ is the unique $^*$\nb-homomorphism with $\pi(\alpha) =
1_{B}$, $\pi(\gamma) = 0$ and $\pi(z) = z$; this is an idempotent
bialgebra morphism.  Its ``image'' is the copy of \(\C(\mathbb{T})\)
generated by~\(z\), its ``kernel'' is the copy of \(A\) generated by
\(\alpha\) and~\(\gamma\).

For \(q=1\), \(B\cong \C(\mathbb{T}\times \SU(2))\) as a
\(\cst\)\nb-algebra, which is commutative.  The representation
on~\(\zesp^2\) combines the standard embedding of~\(\SU(2)\) and the
representation of~\(\mathbb{T}\) mapping~\(z\) to the diagonal
matrix with entries \(z,1\).  This gives a homeomorphism
\(\mathbb{T}\times\SU(2) \cong \nSU(2)\).  So \((B,\Delta_B)\) is
the group~\(\nSU(2)\), written as a semidirect product of~\(\SU(2)\)
and~\(\mathbb{T}\).

For \(q\neq1\), \((B,\Delta_B)\) is the coopposite of the quantum
$\nSU_q(2)$ group described previously by Zhang and Zhao
in~\cite{Zhang-Zhao:Uq2}: the substitutions \(a=\alpha^*\),
\(b=\gamma^*\) and \(D=z^*\) turn our generators and relations into
those in~\cite{Zhang-Zhao:Uq2}, and the comultiplications differ only
by a coordinate flip.

\begin{Thm}
  \label{the:repr_SU_U}
  Let $U\in \M(\mathcal{K}(\mathcal{H})\otimes\C(\mathbb{T}))$ be a
  unitary representation of~$\mathbb{T}$ on a Hilbert
  space~$\mathcal{H}$.  There is a bijection between representations
  of \(\SU_q(2)\) on~\(\mathcal{H}\) and representations
  of~\((B,\Delta_B)\) on~\(\mathcal{H}\) that restrict to the given
  representation on~$\mathbb{T}$.
\end{Thm}

\begin{proof}
  Let $v\in \M(\mathcal{K}(\mathcal{H})\otimes A)$ be a
  unitary representation of~$\SU_q(2)$ on~$\mathcal{H}$.  Since
  $B$ contains copies of $A$ and
  $\C(\mathbb{T})$, we may view $u=vU^*$ as an element of
  $\M(\mathcal{K}(\mathcal{H})\otimes B)$.  The
  $\mathbb{T}$\nb-invariance of~$v$,
  \[
  (\id\otimes\rho^A)(v) = U_{12}^*v_{13}U_{12}
  \]
  and the formula for~$\iota_2$ (which is basically given by the
  action~$\rho^A$) show that
  \[
  U_{12}(\id\otimes\iota_2)(v)U_{12}^* = v_{13}.
  \]
  Using $(\id\otimes\iota_2)(v)=v_{12}$, we conclude that~$u$ is a
  unitary representation of~$(B,\Delta_B)$:
  \[
  (\id\otimes\Delta_B)(u) =
  v_{12}(\id\otimes\iota_2)(v)U^*_{12}U^*_{13}
  = v_{12}U^*_{12}v_{13}U^*_{13} = u_{12}u_{13}.
  \]
  Going back and forth between \(u\) and~\(v\) is the desired
  bijection.
\end{proof}

\end{document}